\documentclass[a4paper,leqno]{article}

\usepackage[centertags]{amsmath}
\usepackage{amsfonts}
\usepackage{amssymb}
\usepackage{amsthm}

\def \Hyp{\mbox{\sl I\kern-.166em H}}   
    \def \Nat{\mbox{\sl I\kern-.166em N}}             
\def \Pr{\mbox{\sl I\kern-.166em P}}

\def \R#1{{\mbox{\sl I\kern-.166em R}}^{#1}}      
\def \Real{\mbox{\sl I\kern-.166em R}}            
\def \Nat{\mbox{\sl I\kern-.166em N}}             
\newtheorem{theorem}{Theorem}[section]

\newtheorem{corollary}[theorem]{Corollary}
\newtheorem{definition}[theorem]{Definition}
\newtheorem{remark}[theorem]{Remark}
\newtheorem{example}[theorem]{Example}

\newcommand{\rid}{\protect\large\bf}

\title{\bf  Geometry of $H$-paracontact metric  manifolds }
\author{ Giovanni Calvaruso and Domenico Perrone
\thanks{Supported by funds of the Universit\'a del Salento and of the M.I.U.R. (within PRIN).
\newline 2000 {\em Mathematics Subject Classification:} {53D10, 53C50, 53C43, 53C25.} 
\newline
{\em Keywords and phrases:}
paracontact metric structures, Reeb vector field, harmonic vector fields, infinitesimal harmonic transformations, Ricci solitons.}
}
\thinspace \thinspace \thinspace \thinspace

\date{}

\begin{document}

\maketitle

\begin{abstract}
{ We introduce and study $H$-paracontact metric manifolds, that is, paracontact metric  manifolds whose Reeb vector field $\xi$ is harmonic. We prove that they are characterized by the condition that $\xi$ is a Ricci eigenvector. We then investigate how harmonicity of the Reeb vector field $\xi$ of a paracontact metric manifold is related to some other relevant geometric properties, like infinitesimal harmonic transformations and paracontact Ricci solitons.}
\end{abstract}


\bigskip\noindent
\section{\rid Introduction}

{ In parallel with contact and complex structures in the Riemannian case, paracontact metric structures were introduced in \cite{KK} in semi-Riemannian settings, as a natural odd-dimensional counterpart to paraHermitian structures. Up to recently, the study of paracontact metric manifolds mainly concerned the special case of paraSasakian manifolds. 

A systemathic study of paracontact metric manifolds started with the paper \cite{Za}, were the Levi-Civita connection, the curvature and a canonical connection (analogue to the Tanaka-Webster connection of the contact metric case) of a paracontact metric manifold have been described.  The technical apparatus introduced in \cite{Za} is essential for further investigations of paracontact metric geometry. Since then, paracontact metric manifolds have been studied under several different points of view. The case when the Reeb vector field satisfies a nullity condition was studied in \cite{CKM}. Conformal paracontact curvature, and its applications, were investigated in \cite{Z2}. In \cite{CIll} the first author studied three-dimensional homogeneous paracontact metric manifolds.

Because of the recent studies of harmonicity conditions in semi-Riemannian geometry, it is a natural problem to investigate when the Reeb vector field of a paracontact metric manifold is a harmonic vector field.
Given a (smooth, oriented, connected) semi-Riemannian manifold $(M,g)$ and a unit vector field $V$ on $M$, the {\em energy} of $V$ is the energy of the corresponding smooth map  $V:(M,g)\rightarrow (T_1M,g^s)$, where $(T_1M,g^s)$ is the unit tangent bundle of $(M,g)$, equipped with the Sasaki metric.  
$V$ is said to be a {\em harmonic vector field} if $V:(M,g)\rightarrow (T_1M,g^s)$ is a critical point for the energy functional restricted to maps defined by unit vector fields. We may refer to the recent monograph  \cite{DrPe} and references therein for an overview on harmonic vector fields. 

The second author \cite{Pe} proved that the Reeb vector field $\xi$ of a contact Riemannian manifold is harmonic if and only if $\xi$ is a Ricci eigenvector. This led to define {\em $H$-contact Riemannian manifolds} as contact metric manifolds, whose Reeb vector field is harmonic.  Since then, $H$-contact Riemannian manifolds have been intensively studied and their relations to other contact geometry properties are now well understood. 

In this paper we introduce the corresponding notion of {\em $H$-paracontact (metric) manifolds}, that is,  paracontact metric manifolds whose Reeb vector field is harmonic. We prove that a paracontact metric manifold is $H$-paracontact if and only if the Reeb vector field is a Ricci eigenvector. This result is not a direct adaptation of its contact Riemannian analogue, because of the deep differences arising between Riemannian and semi-Riemannian settings. In fact, the results proved in \cite{Pe} uses in an essential way the fact that in the Riemannian case, a self-adjoint operator admits an orthonormal basis of eigenvectors, while this property does not hold in semi-Riemannian settings.  

We then investigate the relationship between $H$-paracontact manifolds and some relevant geometric properties, like the Reeb vector field being an infinitesimal harmonic transformation or the paracontact metric structure being a paracontact Ricci soliton.  Under these points of view, the Riemannian case presents some strong rigidity results. However, these results do not hold any more in general semi-Riemannian settings. This makes interesting to study contact semi-Riemannian structures, whose Reeb vector field is an infinitesimal harmonic transformation or determines a Ricci soliton.

{ The paper is organized in the following way. In Section~2 we report some basic information about paracontact metric manifolds and harmonicity properties of vector fields. The characterization of $H$-paracontact metric manifolds in terms of the Ricci operator is proved in Section~3, { where we also prove that the notion of $H$-paracontact manifold is invariant under $D$-homothetic deformations}. In Section~4 we prove that several { classes} of paracontact metric manifolds (paraSasakian and $K$-paracontact manifolds, paracontact $(\kappa,\mu)$-spaces, three-dimensional homogeneous paracontact metric manifolds) are   $H$-paracontact, so showing that the class of $H$-paracontact metric manifolds is rather large. The relationship between $H$-paracontact metric manifolds and paracontact metric manifolds, whose Reeb vector field is $1$-harmonic (equivalently, an infinitesimal harmonic transformation) or whose vector field determines a Ricci soliton, are then investigated in Section~5. Differently from the contact Riemannian case, the class of paracontact metric structures, whose Reeb vector field is an infinitesimal harmonic transformation, is strictly larger than the one of $K$-paracontact structures. 
}

\section{\rid Preliminaries }
\setcounter{equation}{0}

{\small \subsection{Paracontact metric manifolds} }

\medskip\noindent
{ The aim of this  Subsection is to report some basic facts about paracontact metric manifolds. All manifolds are assumed to be connected and smooth. 
We may refer to \cite{KK}, \cite{Za} and references therein for more information about paracontact metric geometry.}

  A $(2n+1)$-dimensional manifold $M$ is said to be a {\it contact manifold} if it admits a global 1-form $\eta$, such that $\eta\wedge (d\eta )^n \not= 0$. Given  such a form $\eta$, there exists a unique vector field $\xi$, called the {\it characteristic vector field} or the {\it  Reeb vector field}, such that $\eta (\xi )=1$ and $d\eta (\xi ,\cdot )=0$. A semi-Riemannian metric $g$ is said to be an  {\it associated metric} if there exists a tensor $\varphi$ of type $(1,1)$, such that 
$$
\eta =  g(\xi ,\cdot ), \qquad d\eta (\cdot ,\cdot )=g(\cdot ,\varphi\cdot ), \qquad \varphi ^2 =  I -\eta\otimes\xi . 
$$
  Then, $(\varphi,\xi,\eta , g)$   (more briefly, $(\eta ,g)$) is called a {\it  paracontact metric    structure}, and $(M, \varphi,\xi,\eta, g)$ a  {\it paracontact    metric  manifold}.

{ As shown in \cite{Za}, any almost paracontact metric manifold $(M{2n+1}, \varphi,\xi,\eta,g)$ admits a {\em $\varphi$-basis}, that is, a local orthonormal basis of the form $\{ \xi , e_1,\dots,e_n, \varphi e_1,\dots, \varphi e_n,\}$, where $\xi,e_1,\dots,e_n$ are space-like vector fields and $\varphi e_1,\dots, \varphi e_n$ are time-like vector fields.}

 We now report some results on the Levi-Civita connection and curvature of a paracontact metric manifold \cite{Za}, which shall be used in the next Section. Let $\nabla$ and $R$ respectively denote the Levi-Civita connection and the corresponding Riemann curvature tensor, taken with the sign convention $$R_{XY} = \nabla_{[X,Y]} - [\nabla_X , \nabla_Y]$$ for all smooth vector fields $X,Y$. Moreover, we shall denote by $\varrho$
the Ricci tensor of type $(0,2)$, by $Q$ the corresponding endomorphism field and by $r$ the scalar curvature.
The tensor
$h=\frac{1}{2} \mathcal L_{\xi} \varphi$,
where $\mathcal L$ denotes the Lie derivative, is symmetric and satisfies  \cite{Za}:
\begin{equation}\label{eq1}
 \nabla \xi = -  \varphi +\varphi h , \quad \nabla _{\xi} \varphi =0 ,
 \quad h\varphi = -\varphi h , \quad  h \xi = 0  , \quad {\rm tr }h=  {\rm tr } h\varphi =0
\end{equation}
and
\begin{align}\label{eq2}
(\nabla_{\varphi X} \varphi)\varphi Y -(\nabla_{X} \varphi)Y= { 2g(X,Y)\xi-\eta(Y)\big(X-h X +\eta(X)\xi\big).}
\end{align}
    


The Ricci curvature of any paracontact metric manifold $(M^{2n+1}, \eta,g)$ satisfies
\begin{equation}\label{eq3}
\varrho (\xi,\xi)=-2n + {\rm tr} h^2 .
\end{equation}


%

{A paracontact metric  manifold $(M,\eta,g)$ is said to be 
\begin{itemize}
\item $\eta$-{\it Einstein} if  its Ricci operator $Q$ is of the form 
\begin{equation}\label{etaE}
Q=aI + b\eta\otimes\xi,
\end{equation} 
where $a,b$ are smooth functions.
\item { a {\em $(\kappa,\mu)$-space} if its curvature tensor satisfies
\begin{equation}\label{km}
R(X,Y)\xi = \kappa (\eta(X)Y-\eta(Y)X) +\mu (\eta(X)hY-\eta(Y)hX),
\end{equation}
for all tangent vector fields $X,Y$, where $\kappa,\mu$ are smooth functions on $M$. }
\item $K$-{\it paracontact} if $\xi$ is a Killing vector field, or equivalently, $h=0$.  
\item {\it paraSasakian} if the paracontact  structure $(\xi, \eta, \varphi, g)$ is {\it normal}, that is, satisfies $[\varphi,\varphi]+2{\rm d}\eta\otimes\xi=0$. This condition is equivalent to  
\begin{displaymath}
(\nabla_X \varphi) Y = -g(X,Y) \xi +  \eta (Y) X. \end{displaymath}
\end{itemize}

\noindent
Any paraSasakian manifold is $K$-paracontact, and the converse also holds when $n=1$, that is, for three-dimensional spaces. {An alternative definition of paraSasakian manifolds, in terms of  {\em cones} over paraK\"ahler manifolds, was given in \cite{AVG}. We also recall that} any paraSasakian manifold satisfies
\begin{align}\label{eq4}
{	R(X,Y)\xi=-( \eta(X)Y-\eta(Y)X),}
\end{align}
so that it is a $(\kappa,\mu)$-space with $k=-1$. To note that, differently from the contact metric case, condition \eqref{eq4} is necessary but not sufficient for a paracontact metric manifold to be  paraSasakian. 
{ This fact was already pointed out in other papers (see for example \cite{CKM}). However, the present authors could not find explicit examples in literature { of paracontact metric manifolds satisfying \eqref{eq4} which are not paraSasakian}. In Subsection~4.3 we shall provide one of such examples in dimension three.}

%


 %

\subsection{\rid Harmonic vector fields}

\medskip\noindent
We now provide some basic information on harmonic vector fields over a semi-Riemannian manifold. For more details, we refer to \cite{GH}, \cite[Chapter~8]{DrPe} and \cite{C3}.

Let $(M,g)$ be an $m$-dimensional semi-Riemannian manifold, $\nabla$ its Levi-Civita connection and $V$ a smooth vector field on $M$. The {\em energy} of $V$ is, by definition, the energy of the corresponding smooth map $V:(M,g) \to (TM,g^s)$, where $g^s$ is the {\em Sasaki metric} (also referred to as the {\em Kaluza-Klein metric} in Mathematical Physics) on the tangent bundle $TM$ of $M$.  If $M$ is compact, then 
$$E(V)=\frac{1}{2}\int _M ({\rm tr}_g V^* g^s) dv
=\frac{m-1}{2} {\rm vol}(M,g) + \frac{1}{2}\int _M ||\nabla V || ^2 dv,$$
while in the non-compact case, one works over relatively compact domains. By the Euler-Lagrange equation,  a vector field $V$ defines a harmonic map from $(M,g)$ to $(TM,g^s)$ if and only if its {\em tension field} $\tau(V)={\rm tr} (\nabla ^2 V)$ vanishes, that is, when 
$$
{\rm tr} [R(\nabla _{\cdot} V,V)\cdot ] =0
\qquad \text{and} \qquad  \bar{\Delta} V= 0.
$$
Here, $\bar{\Delta} V := -\textrm{tr}{ \nabla}^2 V$ is the socalled {\em rough Laplacian} of $V$.
With respect to any local pseudo-orthonormal frame field $\{E_1,..,E_m\}$ on $(M,g)$, with $\varepsilon _i =g(E_i,E_i) =\pm 1$ for all indices $i=1,\dots,m$, we have 
$$
\bar{\Delta} V= \sum _i \varepsilon _i \left(\nabla _{\nabla _{E_i} E_i} V -\nabla _{E_i} \nabla _{E_i} V\right).
$$
If $g$ is Riemannian and $M$ is compact, then parallel vector fields are the only vector fields defining harmonic maps.

Next, for any real constant $r \neq 0$, let $\mathfrak X^r (M)=\{V\in \mathfrak X (M) : ||V||^2=r \}$ denote the set of tangent vector fields of constant lenght $r$. A vector field $V \in \mathfrak X^r(M)$ is said to be {\em harmonic} if it is a critical point for the energy functional $E|_{\mathfrak X^r(M)}$, restricted to vector fields of the same lenght. The Euler-Lagrange equation of this variational condition yields that $V$ is a harmonic vector field if and only if 
\begin{equation}\label{Cd3}
\bar{\Delta} V \quad \text{is collinear to} \quad V.
\end{equation}
This characterization was first obtained, in the Riemannian case, by G. Wiegmink and C.M. Wood (see \cite{DrPe}, p.65). 
 In semi-Riemannian settings, the same argument applies for vector fields of constant lenght, if not light-like \cite{C3}.

Let $T_1 M$ denote the {\em unit tangent sphere bundle} over $M$, and $g^s$ the metric induced on $T_1 M$ by the Sasaki metric of $TM$. Then, the map 
$V: (M,g)\rightarrow (T_{1}M, g^s )$ is harmonic if $V$ is a harmonic vector field and the  additional condition 
\begin{equation}\label{Cd1}
\textrm{tr}[R(\nabla_{\cdot}V ,V)\cdot] =0
\end{equation}
holds. In analogy with the contact metric case \cite{Pe}, we now introduce the following definition.

\begin{definition}
{\em A paracontact  metric manifold $(M,\varphi,\xi,\eta,g)$ is said to be {\em  H-paracontact} if its Reeb vector field $\xi$ is  a harmonic vector field. }
\end{definition}

\noindent
We will show in Theorem \ref{th3} that this notion is also invariant under $D$-homothetic deformations of the paracontact metric structure.

\section{\rid Harmonicity of the Reeb vector field of a { paracontact manifold}}
\setcounter{equation}{0}

In this Section we shall prove the following characterization of { $H$-paracontact metric} manifolds.
\noindent

\begin{theorem}\label{main1}
A  paracontact  metric  manifold is  $H$-paracontact if and only if
the Reeb vector field $\xi$  is an eigenvector of the Ricci operator.
\end{theorem}

\noindent
 The above Theorem~\ref{main1} will be obtained as a consequence of the following result.


\begin{theorem}\label{main2}
Let $(M,\eta ,\xi ,g,\varphi)$ be a $(2n+1)$-dimensional paracontact metric  manifold. Then, 
\begin{equation}\label{eq3.1}
\bar{\Delta} \xi = - 4n \xi - Q\xi  
=  \parallel\nabla \xi\parallel^2 \xi - \textrm{pr}_{\mid\textrm{ker}\eta}{Q\xi},
\end{equation}
where  $\parallel\nabla \xi\parallel^2 = -(2n + \textrm{tr}h^2)$ and $\textrm{pr}_{\mid \textrm{ker} \eta}$ denotes the projection on $\textrm{ker}\,\eta$.
 \end{theorem}
	\begin{proof}
	 Let $M$ be a (2n+1)-dimensional   paracontact  metric  manifold
and $\{ E_1, \dots,E_{2n+1}\}=\{ e_1, \dots,e_{n},  \varphi e_1, \dots,\varphi e_{n}, \xi\}$ a local pseudo-orthonormal $\varphi$-basis, with  $g(e_i,e_i) =-g(\varphi e_i,\varphi e_i)= 1$. We shall use the notation $g(E_i,E_i)=\varepsilon_i =\pm 1$, for $i=1,\dots,2n+1$. We obtain
\begin{align*}
	\bar{\Delta}\xi &= -\sum_{i=1}^{2n+1}\varepsilon_i \left(\nabla_{E_i}\nabla_{E_i}\xi-{ \nabla_{\nabla_{E_i}E_i}\xi }\right)\\
	&= -\sum_{i=1}^{2n+1}\varepsilon_i \left(\nabla_{E_i}\nabla\xi \right){E_i}\\
	&= -\sum_{i=1}^{2n+1}\varepsilon_i \left(\nabla_{E_i}(-\varphi +\varphi h ) \right){E_i}\\
	&=   \sum_{i=1}^{2n+1}\varepsilon_i \left( \nabla_{E_i}\varphi \right)E_i -  \sum_{i=1}^{2n+1}\varepsilon_i \left( \nabla_{E_i}\varphi\, h \right)E_i \\ 
	&=  \textrm{tr}   \nabla \varphi - \textrm{div}(\varphi\, h).
\end{align*}
Using formula \eqref{eq2}, we have
\begin{align}\label{eq3.1'}
 \textrm{tr}   \nabla \varphi =  (\nabla_{e_i} \varphi){e_i}-(\nabla_{\varphi e_i} \varphi){\varphi e_i} +(\nabla_{\xi} \varphi){\xi}= -2n\xi
\end{align}
{ and so,} 
\begin{align}\label{eq3.2}
	\bar{\Delta}\xi &= -2n \xi - \textrm{div}(\varphi\,h).
\end{align}
{By Equation \eqref{eq1}, $\nabla \xi= -\varphi +\varphi h$. Differentiating, we get}
\begin{align}\label{eq3.3}
R(X,Y)\xi= (\nabla_{X} \varphi)Y -(\nabla_{Y} \varphi)X- (\nabla_{X} \varphi h)Y+ (\nabla_{Y} \varphi h)X,
\end{align}
from which we deduce that the Ricci curvature $\varrho(X,\xi)$ is given by 
\begin{align}\label{eq3.4}
\varrho(X,\xi)&= \textrm{tr} R(X,\cdot )\xi  \\&= \sum_{i=1}^{{2n}}\varepsilon_i g\left(R(X,E_i)\xi,E_i\right) \nonumber \\
             &= \sum_{i=1}^{{2n}}\varepsilon_i g((\nabla_X \varphi)E_i,E_i)- \sum_{i=1}^{{2n}}\varepsilon_i g((\nabla_{E_i} \varphi)X,E_i) \nonumber \\
	 & \qquad -\sum_{i=1}^{{2n}}\varepsilon_i g((\nabla_X \varphi\,h)E_i,E_i)   + \sum_{i=1}^{{2n}}\varepsilon_i g((\nabla_{E_i} \varphi\,h)X,E_i). \nonumber 
\end{align}
By direct calculation, we find 
\begin{align}%
&\sum_{i=1}^{{2n}}\varepsilon_i g((\nabla_X \varphi)E_i,E_i)=0 \label{tr0},\\[6 pt]
&\sum_{i=1}^{{2n}}\varepsilon_i g((\nabla_X \varphi\,h)E_i,E_i)=\textrm{tr} \nabla_X(\varphi\,h) = \nabla_X\textrm{tr}(\varphi\,h)=0 ,\nonumber \\[6 pt]
&\sum_{i=1}^{{2n}}\varepsilon_i g((\nabla_{E_i} \varphi\,h)X,E_i) = g(\textrm{div}(\varphi\,h), X) \nonumber  \end{align}
and, taking into account	$\textrm{tr}\nabla\varphi =-2n\xi$, 
\begin{equation}\label{trace}
\sum_{i=1}^{{2n}}\varepsilon_i g((\nabla_{E_i} \varphi)X,E_i) = 2n\, \eta(X).
\end{equation}
{ We then replace into \eqref{eq3.4} and  we obtain}
\begin{align}\label{eq3.5}
	\varrho (X,\xi)= - 2n \eta(X) + g(\textrm{div}(\varphi\,h), X).
\end{align}
On the other hand, $\varrho(\xi,\xi)= -2n + \textrm{tr}h^2= -4n -\left\|\nabla\xi\right\|^2$. So, Equations~\eqref{eq3.2} and \eqref{eq3.5} yield
\begin{align*}
	\bar{\Delta}\xi &= -4n  \xi - Q\xi\\
	&= -(2n+ \textrm{tr}h^2) \xi - (Q\xi)_{|\textrm{ker}\eta}\\
	&=   \left\|\nabla\xi\right\|^2 \xi - (Q\xi)_{|\textrm{ker}\eta} 
	\end{align*}
and this ends the proof.
	\end{proof}

\noindent
 
\begin{remark}
{ {\em The contact metric analogues of the above  Theorems~\ref{main1} and \ref{main2} were proved in \cite{Pe}. The proof of these contact Riemannian results
used in an essential way the fact that the tensor $h$, being self-adjoint, admits an orthonormal basis of eigenvectors. The lack of such information in the paracontact metric case required a different approach to the proof of the above  Theorem~\ref{main2}.}
}
\end{remark}

As an immediate consequence of Theorem~\ref{main1} and the definiton of $\eta$-Einstein manifolds, we have the following.

\begin{corollary}\label{cor3.4}
$\eta$-Einstein  paracontact metric  manifolds are $H$-paracontact.
 \end{corollary}

{We now prove that any $D$-homothetic deformation of a $H$-paracontact metric structure is again $H$-paracontact. This fact shows that the harmonicity of the Reeb vector field is rather  natural  for paracontact metric manifolds, and permits to build new examples of $H$-paracontact metric structures from the known ones.
 
Given a paracontact metric structure $(\eta,g,\xi,\varphi)$, its {\em $D$-homothetic deformation}, determined by any real constant $t \neq 0$, is the new paracontact metric structure $(\eta_t,g_t,\xi_t,\varphi_t)$,  defined by 
\begin{equation}\label{DH}
\eta_{t}=t\eta , \qquad \xi_{t}= t^{-1}\xi , \qquad \varphi =\varphi, \qquad g_{t}=tg + \varepsilon  t(t-1)\eta\otimes\eta .
\end{equation}
(see \cite{Za},\cite{CKM}). In \cite{CKM}, the relationships between the Levi-Civita connections $\nabla$ and $\nabla _t$ and curvature tensors $R$ and $R_t$ of $g$ and $g_t$ respectively were investigated. In particular, { by Proposition~3.6 in \cite{CKM}, rewritten for our sign convention  of the curvature, we have}
\begin{align*}
t R_t (X,Y)\xi_t =& R(X,Y)\xi +(t-1)^2 (\eta(Y)X -\eta(X)Y)\\[4 pt]
&+(t-1)\big( (\nabla _X \varphi)Y -(\nabla _Y \varphi)X +\eta(Y)(X-hX)-\eta(X)(Y-hY)\big) . \nonumber
\end{align*}
When both $X,Y$ belong to ker$\eta$, the above equation reduces to
\begin{align}\label{RRt}
t R_t (X,Y)\xi_t =& R(X,Y)\xi +(t-1)\big( (\nabla _X \varphi)Y -(\nabla _Y \varphi)X .
\end{align}
Let now $\{E_1,\dots,E_{2n+1}\}=\{e_1,\dots,e_n,\varphi e_1,\dots, \varphi e_n,\xi \}$ be a local $\varphi$-basis for $(\eta,g\xi,\varphi)$. To note that ker$\eta_t$=ker$\eta$ and $g_t =t g$ on ker$\eta$. Therefore, $\{\frac{1}{\sqrt{t}}e_1,\dots,\frac{1}{\sqrt{t}}e_n,\frac{1}{\sqrt{t}}\varphi e_1,\dots, \frac{1}{\sqrt{t}}\varphi e_n,\xi_t \}$ is a local basis of vector fields, pseudo-orthonormal with respect to $g_t$. We can now calculate the Ricci tensor $\varrho_t (X,\xi_t)$, for any vector field $X\in$ker$\eta$, by contraction of \eqref{RRt}. Taking into account Equations~\eqref{tr0} and \eqref{trace}, we get
\begin{align*}
\varrho_t (X,\xi_t) =& \frac{1}{t}\sum _{i=1}^{2n} \varepsilon _i g_t (R_t (X,E_i)\xi_t,E_i) = \sum _{i=1}^{2n} \varepsilon _i g (R_t (X,E_i)\xi_t,E_i) \\[4 pt]
=& \frac{1}{t}\sum _{i=1}^{2n} \varepsilon _i g (R (X,E_i)\xi +\frac{t-1}{t}\sum _{i=1}^{2n}g((\nabla_{X}\varphi) E_i,E_i) -\frac{t-1}{t}\sum _{i=1}^{2n}g((\nabla_{E_i}\varphi) X,E_i)  \\[.4 pt]
=& \frac{1}{t} \varrho(X,\xi)-\frac{t-1}{t}\cdot 2n \eta(X) \\[4 pt] 
=& \frac{1}{t} \varrho(X,\xi). 
\end{align*}
Thus, the property \lq\lq $\xi$ is an eigenvector of the Ricci operator\rq\rq  \ is invariant under $D$-homothetic deformations, and  Theorem~\ref{main1} yields the following result.  

\begin{theorem}\label{th3} 
The class of  $H$-contact semi-Riemannian manifolds is invariant under  $D$-homo\-thetic deformations.
\end{theorem}

\section{\rid Examples}

We shall now investigate the relationships among the class of $H$-paracontact spaces and some relevant  classes of paracontact metric manifolds. 

\subsection{$K$-paracontact and paraSasakian manifolds}

We start considering the $K$-paracontact case, for which we shall prove the following.

\begin{theorem}\label{teo.1}  The Ricci operator of a $K$-paracontact metric manifold satisfies
\begin{align*}
	Q\xi = -2n \xi .
	\end{align*}
Hence, $K$-paracontact (in particular, paraSasakian) manifolds are $H$-paracontact. Moreover, the Reeb vector field $\xi$ of any $K$-paracontact metric manifold $(M,\eta,g)$ defines a harmonic map
$\xi: (M,g) \rightarrow (T_1 M, g^s)$.
\end{theorem}
		
\begin{proof} 
It follows from Equation~\eqref{eq3.3} that the curvature tensor of a $K$-paracontact metric manifold satisfies
\begin{align}\label{eq3.7}
	R(Y,Z,\xi,X)= g((\nabla_Y \varphi)Z, X)- g((\nabla_Z \varphi)Y, X) .
\end{align}
Then, using the first Bianchi identity, again \eqref{eq3.3}, and hence \eqref{eq3.7}, we get
\begin{align}\label{eq3.8}
	R(\xi,X,Y,Z)= g((\nabla_X \varphi)Z, Y).
\end{align}
Let now $\{ e_1, \dots,e_{n},  \varphi e_1, \dots,\varphi e_{n}, \xi\}$ be a local pseudo-orthonormal $\varphi$-basis, with  $g(e_i,e_i) = -g(\varphi e_i,\varphi e_i)= 1$. Using the formulae \eqref{eq3.8} and   \eqref{eq3.1'}, we obtain
\begin{align*}
\varrho(\xi, Y)&= \textrm{tr} R(\xi,\cdot, Y,\cdot) \\ &= \sum_{i=1}^{n} g\left(R(\xi,e_i,Y,e_i\right)- \sum_{i=1}^{n} g\left(R(\xi,\varphi e_i,Y,\varphi e_i\right)\\
	 & =\sum_{i=1}^{n}  g((\nabla_{e_i} \varphi)e_i, Y)- g((\nabla_{\varphi e_i} \varphi)\varphi e_i, Y)\\
	 & = g(\textrm{tr}\nabla \varphi,Y) \\ &=-2n\eta(Y),
\end{align*}
that is, $Q\xi=-2n\xi$. Hence, $(M,\eta,g)$ is $H$-contact.  In order to conclude that $\xi: (M,g) \rightarrow (T_1 M, g^s)$ is a harmonic map, it then suffices to prove that ${\rm tr} [R(\nabla\xi,\xi)\cdot ] = 0$.

Since $M$ is $K$-contact, the first equation in \eqref{eq1} reduces to $\nabla \xi=-\varphi$.
{ With respect to the above local pseudo-orthonormal $\varphi$-basis  $\left\{e_i,\varphi e_i, \xi \right\}$, using $\nabla \xi=-\varphi$ and the first Bianchi identity, we obtain }
\begin{eqnarray*}
	{\rm tr} [R(\nabla\xi,\xi)\cdot ] &=&-{\rm tr} [R(\varphi\cdot ,\xi)\cdot ] \\
	  &=&-\sum^{n}_{i=1}  R(\varphi e_i,\xi)e_i +\sum^{n}_{i=1} R(\varphi^2 e_i,\xi)\varphi e_i \\
	&=&-\sum^{n}_{i=1}R(e_i,\varphi e_i)\xi .
	\end{eqnarray*}
On the other hand, by equations~\eqref{eq3.3}  and  \eqref{eq2} with $h=0$, we get
\begin{eqnarray*}	
R(e_i,\varphi e_i)\xi	= \big(\nabla_{e_i}\varphi\big) \varphi e_i  -    \big(\nabla_{\varphi e_i}\varphi\big)e_i = \big(\nabla_{\varphi e_i}\varphi\big) \varphi^2 e_i  -    \big(\nabla_{\varphi e_i}\varphi\big)e_i = 0.
\end{eqnarray*}
So,   we conclude that ${\rm tr} [R(\nabla\xi,\xi)\cdot ]=0$ and this ends the proof. 
\end{proof}

{
Consider $\mathbb R^{2n+2}$, equipped with the standard paracomplex structure $\mathbb I$ and flat metric $g$ of neutral signature. Then, any non-degenerate hypersurface in $(\mathbb R^{2n+2},\mathbb  I, g)$  inherits an integrable para-contact hermitian structure \cite{Z2}. In particular, a standard example of paraSasakian manifold is given by the Hyperboloid 
$$HS^{2n+1}:=\{(x_0,y_0,\dots,x_n,y_n) | x_0^2 +\dots +x_n ^2-y_0^2- \dots y_n^2 =1 \},$$
with the natural para-CR structure induced by its embedding in $(\mathbb R^{2n+2},\mathbb I,g)$. In this case,
$$\eta =\sum _{j=0}^n (y_jdx_j-x_j dy_j), \quad \xi=\sum_{j=0} ^n (x_j \frac{\partial}{\partial y_j}+y_j\frac{\partial}{\partial x_j}), \quad \varphi=\mathbb I |_{HS^{2n+1}}, \quad g |_{HS^{2n+1}\times HS^{2n+1}}$$
is a paraSasakian structure. By the above Theorem~\ref{teo.1}, {\em the Reeb vector field $\xi$ of the canonical  paraSasakian structure of $HS^{2n+1}$ defines a harmonic map into its unit tangent sphere bundle.}
}

\begin{remark}
{\em $K$-contact Riemannian manifolds are {\em characterized} by the Ricci curvature condition $Q\xi=2n\xi$. As we proved in the above Theorem~\ref{teo.1}, $K$-paracontact manifolds satisfy the corresponding condition $Q\xi=-2n\xi$, but the converse does not hold. In fact, if $Q\xi=-2n\xi$, then by \eqref{eq3} we find tr$h^2=0$. However, for a paracontact metric manifold, the tensor $h$ needs not to be diagonalizable. Consequently, tr$h^2=0$ does not imply that a paracontact metric manifold is $K$-paracontact. Explicit examples of paracontact metric manifolds with tr$h^2=0$ (indeed, with $h^2=0$) but $h \neq 0$ will be given { in the next subsection~\ref{3Dex}.}}
\end{remark}

\subsection{$(\kappa,\mu)$-paracontact metric manifolds}

We now consider paracontact metric manifolds, whose Reeb vector field satisfies the nullity condition \eqref{km}.  By contraction of \eqref{km},  it is easily seen that the Ricci operator of a $(\kappa,\mu)$-paracontact metric manifold satisfies $Q\xi=2n\kappa \xi$ (see also \cite{CKM}, p.670).
Therefore, Theorem~\ref{main1} implies at once that $(\kappa,\mu)$-paracontact metric manifolds are $H$-paracontact.

Next, the following formula holds for $(\kappa,\mu)$-paracontact metric manifolds with $\kappa \neq -1$ (see \cite{CKM}, pp.682 and 690, rewritten here for our sign convention on the curvature tensor):
\begin{equation}\label{kno1}
\begin{array}{rcl}
R(X,Y)hZ-hR(X,Y)Z&\hspace{-3mm}=\hspace{-3mm}&\big(\kappa(\eta(Y)g(hX,Z)-\eta(X)g(hY,Z)) \\[4 pt] &\hspace{-3mm} &\hspace{-3mm} +\mu (\kappa+1)(\eta(Y)g(X,Z)-\eta(X)g(Y,Z)) \big)\xi \\[4 pt]
&\hspace{-3mm} &\hspace{-3mm}+\kappa \big( g(Y,\varphi Z)\varphi hX -g(X,\varphi Z)\varphi hY+g(Z,\varphi hY)\varphi X-g(Z,\varphi hX)\varphi Y  \big)\\[4 pt]
&\hspace{-3mm} &\hspace{-3mm} +\eta(Z)\big(\eta(Y)hX-\eta(X)hY \big) +\mu \big((\kappa+1)\eta(Z)(\eta(Y)X-\eta(X)Y) \big)\\[4 pt]
&\hspace{-3mm} &\hspace{-3mm}+2\mu g(X,\varphi Y)\varphi h Z,
\end{array}
\end{equation}
for all tangent vector fields $X,Y,Z$.  
Let now $(M,\varphi,\xi,\eta,g)$ denote any $(\kappa,\mu)$-paracontact metric manifold  with $\kappa \neq 1$, and consider a $\varphi$-basis $\{\xi,e_1,\dots,e_n,\varphi e_1,\dots,\varphi e_n\}$. Using the first equation in \eqref{eq1} and the first Bianchi identity, we find 
\begin{align*}
-{\rm tr}[R(\nabla . \, \xi,\xi)\cdot] &= -\sum _{i=1}^{n} R(-\varphi e_i+\varphi h e_i, \xi)e_i +\sum _{i=1}^{n} R(- e_i+\varphi h \varphi e_i, \xi) \varphi e_i \\[4 pt] &=  -\sum _{i=1}^{n} R(-\varphi e_i+\varphi h e_i, \xi)e_i +\sum _{i=1}^{n} R(- e_i- h e_i, \xi) \varphi e_i \\[4 pt] &=
 \sum _{i=1}^{n} R(\xi, e_i) \varphi h e_i +\sum _{i=1}^{n} R(\xi, \varphi e_i) h e_i \\[4 pt] &=
 -\sum _{i=1}^{n} R(\xi, e_i) h\varphi e_i -\sum _{i=1}^{n} R(\xi, \varphi e_i) h \varphi ^2 e_i .
\end{align*}
On the other hand, by \eqref{kno1} we have
$$R(\xi,E_j)h\varphi E_j =hR(\xi,E_j)\varphi E_j +\kappa g(hE_j,\varphi E_j)\xi,$$
for any $E_j \in \{e_1,\dots,e_n,\varphi e_1,\dots,\varphi e_n\}$. Therefore, we conclude that
\begin{align*}
-{\rm tr}[R(\nabla . \, \xi,\xi)\cdot] &= 
 -\sum _{i=1}^{n} R(\xi, e_i) h\varphi e_i -\sum _{i=1}^{n} R(\xi, \varphi e_i) h \varphi ^2 e_i \\[4 pt] &= -\sum _{i=1}^{n} \big( hR(\xi, e_i) \varphi e_i +\kappa g(he_i,\varphi e_i)\xi -hR(\xi, \varphi e_i)e_i -\kappa g(h \varphi e_i,\varphi ^2 e_i)\xi \big)  \\[4 pt] &= -\sum _{i=1}^{n} \big( hR(\xi, e_i) \varphi e_i +\kappa g(he_i,\varphi e_i)\xi -hR(\xi, \varphi e_i)e_i -\kappa g(h e_i,\varphi  e_i)\xi \big)
 \\[4 pt] &= -\sum _{i=1}^{n} h\big( R(\xi, e_i) \varphi e_i  -R(\xi, \varphi e_i)e_i \big)  
 \\[4 pt] &= -\sum _{i=1}^{n} h\big( R(e_i, \varphi e_i) \xi \big) =0,  
\end{align*}
by the $(\kappa,\mu)$-nullity condition. Therefore, we proved the following result.

\begin{theorem}\label{teo3.4}
$(\kappa,\mu)$-paracontact metric manifolds are $H$-paracontact. Moreover, whenever $\kappa \neq -1$, the Reeb vector field of a paracontact $(\kappa,\mu)$-space also defines a harmonic map into its unit tangent sphere bundle.
\end{theorem}
  
It should be noted that paraSasakian manifolds are $(\kappa,\mu)$-paracontact metric manifolds with $\kappa =-1$, but not conversely. It is interesting to investigate non-paraSasakian paracontact $(\kappa,\mu)$-spaces with $\kappa =-1$ \cite{CKM}, also in order to decide whether their Reeb vector field defines a harmonic map into the unit tangent sphere bundle.

\subsection{Three-dimensional homogeneous paracontact metric manifolds}\label{3Dex}

In \cite{CIll}, the first author obtained the complete classification of three-dimensional homogeneous paracontact metric manifolds. The classification result is the following.

\begin{theorem}{\bf \cite{CIll}}\label{mainTh}
A simply connected complete homogeneous paracontact metric three-manifold is isometric to a Lie group $G$ with a left-invariant paracontact metric structure $(\varphi,\xi,\eta,g)$. More precisely, one of the following cases occurs: 

\medskip\noindent 
(i) If $G$ is unimodular, then the Lie algebra of $G$ is one of the following:  
\begin{itemize}
\item[(1)] $\qquad \mathfrak{g}_2 : \quad [\xi,e]=-\gamma e +\beta \varphi e, \quad 
[\xi, \varphi e]=\beta e +\gamma \varphi e , \quad  \left[e,\varphi e \right]=2 \xi, \; with \;\; \gamma \neq 0.$

\end{itemize}%
In this case, $G$ { is either the identity component of $O(1,2)$, or}  $\widetilde{SL}(2,\mathbb R)$.
\begin{itemize}
\item[(2)] $\qquad \mathfrak{g}_3 : \quad \left[\xi ,e  \right]=-\gamma \varphi  e, \quad 
\left[ \xi, \varphi e\right]=-\beta e, \quad \left[e,\varphi e\right]=2 \xi .$
\end{itemize}
In this case, $G$ is 
\begin{itemize} 
\item[(2a)] the identity component of $O(1,2)$ or $\widetilde{SL}(2,\mathbb R)$ if either $\beta,\gamma>0$ or $\beta,\gamma<0$;
\item[(2b)] $\widetilde{E}(2)$  if either $\beta>0=\gamma$ or $\beta=0>\gamma$;
\item[(2c)] $E(1,1)$ if either $\beta<0=\gamma$ or $\beta=0<\gamma$;
\item[(2d)] either $SO(3)$ or $SU(2)$ if $\beta>0$ and $\gamma <0$;
\item[(2e)] the Heisenberg group $H_3$ if $\beta=\gamma =0$.
\end{itemize}
\begin{itemize}
\item[(3)] $\quad \mathfrak{g}_4 : \; \left[\xi ,e \right]=- e + (2 \varepsilon - \beta) \varphi e, \quad \left[ \xi ,\varphi e \right]=-\beta e +\varphi e, \quad
\left[e,\varphi e\right]=2 \xi , \; with \; \varepsilon=\pm 1.$
\end{itemize}
In this case, $G$ is
\begin{itemize} 
\item[(3a)] the identity component of $O(1,2)$ or $\widetilde{SL}(2,\mathbb R)$ if $\beta \neq \varepsilon$;
\item[(3b)] $\widetilde{E}(2)$ if $\beta=\varepsilon=1$;
\item[(3c)] $E(1,1)$ if $\beta=\varepsilon=-1$.
\end{itemize}   
\medskip\noindent 
(ii) if $G$ is non-unimodular, then Lie algebra of $G$ is one of the following:  
\begin{itemize}
\item[(4)] $\qquad \mathfrak{g}_5, \mathfrak{g}_6 : \quad  \left[\xi ,e \right]=\left[ \xi ,\varphi e \right]=0, \quad \left[e, \varphi e \right]=2 \xi +\delta e, \; with \; \delta \neq 0.$
\item[(5)] $\qquad \mathfrak{g}_7 : \quad  [\xi ,e]=-[\xi,\varphi e]=-\beta (e + \varphi e),  \quad [e, \varphi e]=2 \xi+\delta (e + \varphi e), \; with \; \delta \neq 0.$
\end{itemize}
\end{theorem}

Notations $\mathfrak{g}_2-\mathfrak{g}_7$ for Lie algebras listed in Theorem~\ref{mainTh} refer to the classification of all three-dimensional Lorentzian Lie groups, obtained in \cite{C}.

In the symmetric case, such a paracontact homogeneous three-manifold is either flat or of constant sectional curvature $-1$. These cases are included in the classification given in Theorem~\ref{mainTh} above. In fact, in case {\em (2a)} with $\alpha=\beta=\gamma=2$, unimodular Lie groups $O(1,2)$ or $\widetilde{SL}(2,\mathbb R)$ have constant sectional curvature $-1$, while in case {\em (2b)} with $\alpha=\beta=2$, the unimodular Lie group $\widetilde{E}(2)$ is flat. 

Tensor $h=(1/2)\mathcal{L}_{\xi} \varphi$ of all examples listed in Theorem~{\ref{mainTh}} can be easily deduced from the above Lie brackets. Moreover, the curvature and the Ricci tensor of any left-invariant Lorentzian structure over a three-dimensional Lie group was completely described in \cite{C2}. 
In particular, describing the Ricci operator with respect to the pseudo-orthonormal basis $\{e_1,e_2,e_3\}=\{\xi,e,\varphi e\}$, we get: 

\smallskip
For case {\em (1)}:  
\begin{equation}\label{curv22}
\left\{\begin{array}{l}  h e = \gamma \varphi e, \\[2 pt] h \varphi e = -\gamma e, \end{array}\right.
\qquad Q= \left(
\begin{array}{ccc}
   -2-2\gamma ^2 & 0 & 0 \\
   0  & 2 -2 \beta & \gamma(2 -2\beta) \\
   0  & -\gamma(2 -2\beta)  & 2 -2 \beta 
\end{array}
\right).
\end{equation}

For case {\em (2)}:  
\begin{equation}\label{curv27}
\left\{\begin{array}{l} he=-\frac 12 (\beta-\gamma) e, \\[4 pt] h\varphi e =\frac 12 (\beta-\gamma)\varphi e, \end{array}\right. \qquad Q= \left(
\begin{array}{ccc}
   -2+\frac 12 (\beta-\gamma)^2 & 0 & 0 \\
   0  &  \frac 12 ((2-\gamma)^2-\beta^2)  & 0 \\
   0  & 0 & \frac 12 ((2-\beta)^2-\gamma^2)
\end{array}
\right).
\end{equation}

For case {\em (3)}:  
\begin{equation}\label{curv31}
\left\{\begin{array}{l} he=\varepsilon e+ \varphi e, \\[2 pt] h\varphi e =-e -\varepsilon \varphi e, \end{array}\right. \qquad \quad Q = \left(
\begin{array}{ccc}
   -2  & 0 & 0 \\
   0  &  4+2\eta(2 -\beta)-2 \beta  & 2(1 +\eta -\beta) \\
   0  & -2(1+\eta -\beta) &  -2 \beta+2\eta \beta
\end{array}
\right). 
\end{equation}
%
%

For case {\em (4)}:  
\begin{equation}\label{curv35}
h=0, \qquad Q = \left(
\begin{array}{ccc}
   -2  & 0 & 0 \\
   0  &   \delta ^2 +2   & 0 \\
   0  & 0 &  \delta ^2 + 2
\end{array}
\right).
\end{equation}
%

For case {\em (5)}:  
\begin{equation}\label{curv43}
\left\{\begin{array}{l}  he=\beta(e+ \varphi e), \\[2 pt] h\varphi e =-\beta(e +\varphi e), \end{array}\right. \qquad Q= \left(
\begin{array}{ccc}
   -2  & 0 & 0 \\
   0 &  2-2\beta  & 2\beta \\
   0  & -2\beta & 2+2\beta 
\end{array}
\right).
\end{equation}
Thus,  in all the above cases, $\xi=e_1$ is a Ricci eigenvector. Hence, by Theorem~\ref{main1}, $\xi$ is harmonic. Indeed, we can prove the following stronger result.  

\begin{theorem}\label{3Dharm}
The Reeb vector field of any three-dimensional homogeneous paracontact metric manifold defines a harmonic map into the unit tangent sphere bundle. In particular, all three-dimensional homogeneous paracontact metric manifolds are $H$-para\-contact. 
\end{theorem} 
  
\begin{proof}
We already concluded by \eqref{curv22}-\eqref{curv43} that $\xi=e_1$ is a Ricci eigenvector. So, all the above examples are $H$-paracontact, and it suffices to check the additional condition ${\rm tr} [R(\nabla _{\cdot} \xi,\xi)\cdot ] =0$. To note that in case {\em (4)}, $h=0$ and the conclusion follows from Theorem~\ref{teo.1}.

Consider a three-dimensional paracontact metric manifold $(M,\varphi,\xi,\eta,g)$ and a local $\varphi$-basis $\{\xi,e,\varphi e\}$. Taking into account $h\varphi=-\varphi h$ and the first equation in \eqref{eq1}, one has
\begin{align}
{\rm tr} [R(\nabla _{\cdot} \xi,\xi)\cdot ] &= R(\nabla _e \xi, \xi)e -R(\nabla _{\varphi e} \xi, \xi)\varphi e \label{tr3D} \\[4 pt] 
&= R(\xi, \varphi e)e -R(\xi, \varphi h e ) e - R(\xi, e)\varphi e -R(\xi, h e)\varphi e . \nonumber
\end{align}
The curvature tensor of three-dimensional left-invariant paracontact metric structures listed in Theorem~{\ref{mainTh}} can be deduced either by direct calculation, or by comparison with the more general formulae obtained in \cite{C2} for the curvature of three-dimensional Lorentzian Lie groups. 

For case {\em (1)}, we find 
  
$$
\begin{array}{ll}  
R(\xi ,e) e = -(2+\gamma^2) \xi, & \quad  R(\xi ,e) \varphi e = 2\gamma (1+\beta) \xi, \\[4 pt] 
R(\xi ,\varphi e) \varphi e = (2+\gamma^2 )\xi , & \quad R(\xi ,\varphi e) e = 2\gamma (1+\beta)\xi . \end{array}
$$
Then, taking into account the description of $h$ given in \eqref{curv22}, from \eqref{tr3D} we get
\begin{align*} {\rm tr} [R(\nabla _{\cdot} \xi,\xi)\cdot ] &= R(\xi, \varphi e)e -\gamma R(\xi, e ) e - R(\xi, e)\varphi e -\gamma R(\xi, \varphi e)\varphi e \\[4 pt] &=
2\gamma (1+\beta)\xi +\gamma (2+\gamma^2)\xi-2\gamma (1+\beta)\xi -\gamma (2+\gamma^2)\xi \\&=0.
\end{align*}
The calculations for the remaining cases are similar to the above one. It suffices to apply \eqref{tr3D}, using the description of tensor $h$ given in equations \eqref{curv27},\eqref{curv31} and \eqref{curv43}, and 
the following curvature equations:
$$
\begin{array}{lll}  \text{For case {\em (2)}:} & \quad 
R(\xi ,e) e = \frac 14 (4\beta-\beta^2 -4+3\gamma ^2-4\gamma-2\beta\gamma) \xi, &  R(\xi , e) \varphi e = 0, \\[4 pt] & \quad
R(\xi ,\varphi e) \varphi e = \frac 14 (4-4\gamma+\gamma^2 -3\beta ^2+4\beta+2\beta\gamma) \xi , &  R(\xi ,\varphi e) e = 0. \\[8 pt]
  \text{For case {\em (3)}:} & \quad 
R(\xi ,e) e = (1+2\varepsilon -2\varepsilon  \beta ) \xi, &  R(\xi , e) \varphi e = 2(1+\varepsilon -\beta)\xi , \\[4 pt] & \quad
R(\xi ,\varphi e) \varphi e = (3+2\varepsilon  -2\varepsilon \beta ) \xi , &  R(\xi ,\varphi e) e = 2(1+\varepsilon -\beta)\xi . \\[8 pt]
\text{For case {\em (5)}:} & \quad 
R(\xi ,e) e = - (1+2\beta) \xi, &  R(\xi , e) \varphi e = 2\beta \xi, \\[4 pt] & \quad
 R(\xi ,\varphi e) \varphi e = (1-2\beta ) \xi , &  R(\xi ,\varphi e) e = 2\beta \xi. \end{array}
$$
In all the above cases, a straightforward calculation yields ${\rm tr} [R(\nabla _{\cdot} \xi,\xi)\cdot ]=0$. So, $\xi$ defines a harmonic map into the unit tangent sphere bundle.
\end{proof}

{ We end this Section by pointing out the following

\begin{example}{\bf A nonSasakian paracontact metric manifold satisfying \eqref{eq4}.} \\
{\em Consider the three-dimensional left-invariant paracontact metric structure listed in case {\em (3a)} of Theorem~\ref{mainTh} in the special case when $\beta=\varepsilon+1$, that is, 
$$ \left[\xi ,e \right]=- e + (\varepsilon - 1) \varphi e, \qquad \left[ \xi ,\varphi e \right]=-(\varepsilon +1) e +\varphi e, \qquad \left[e,\varphi e\right]=2 \xi , \; \text{with} \; \varepsilon=\pm 1.
$$
We already proved in equation~\eqref{curv31} that $he=\varepsilon e+ \varphi e$ and 
$h\varphi e =-e -\varepsilon \varphi e$. Therefore, $h \neq 0$ and so, this paracontact metric structure is not paraSasakian. On the other hand, calculating the curvature tensor (or equivalently, using the formulas proved in \cite{C2} for the curvature), we easily get 
$$\begin{array}{l}
R(e, \varphi e)\xi=0, \\[4 pt]
R(\xi,  e)\xi= -(2\varepsilon \beta -1-2\varepsilon )e-2(1+\varepsilon -\beta)\varphi e = -e=-\eta(\xi)e, \\[4 pt]
R(\xi, \varphi e)\xi=2(1+\varepsilon -\beta)e -(3+2\varepsilon -2\varepsilon \beta)\varphi e=-\varphi e =\eta(\xi) \varphi e,
\end{array}$$
from which it follows at once that equation~\eqref{eq4} holds, since $R$ and $\eta$ are tensors. Thus, this paracontact metric manifold is an explicit example of a paracontact non-paraSasakian metric manifold, satisfying \eqref{eq4}.     
}\end{example}
}

\section{\rid Paracontact infinitesimal harmonic transformations and Ricci \\ solitons}
\setcounter{equation}{0}

Let $(M^n,g)$ be a semi-Riemannian manifold and $f : x \mapsto x'$ a point transformation in $(M, g)$. If $\nabla (x)$ denotes the Levi-Civita connection at $x$ and $\nabla '(x)$ is obtained bringing back $\nabla (x')$ to $x$ by $f^{-1}$ \cite{SS2}, the {\em Lie difference} at $x$ is defined as $\nabla'(x)-\nabla(x)$. The map $f$ is said to be {\em harmonic} if ${\rm tr}(\nabla'(x)-\nabla(x))=0$.

Consider now  a vector field $V$ on $M$ and the local one-parameter group of infinitesimal point transformations $f_t$ generated by $V$.  The Lie derivative $L_{V} \nabla$ then corresponds to $\nabla'(x)-\nabla(x)$, where $\nabla'(x)=f^* _t (\nabla(x'))$, and $V$ generates a group of harmonic transformations if and only if $${\rm tr}(L_V \nabla)=0.$$ In this case, $V$ is said to be an {\em infinitesimal harmonic transformation} {\cite{No1},\cite{SS1}}.

Infinitesimal harmonic transformations also occur as critical points for a suitable energy functional. In fact, if $g^c$ denotes the {\em complete lift metric} of $g$ to $TM$,  which is of neutral signature $(n,n)$, a vector field $V$ on $M$ defines a harmonic section $V: (M, g) \to (TM, g^c)$ if and only if $V$ is an infinitesimal harmonic transformation \cite{No1}. For this reason, infinitesimal harmonic transformations are also called {\em $1$-harmonic vector fields}, because this harmonicity property is equivalent to the vanishing of the linear part of the tension field of the local one-parameter group of infinitesimal point transformations \cite{DTV}. 
A vector field $V$ is an infinitesimal harmonic transformation if and only if $\bar \Delta V = Q V$ (see for example \cite{CP},\cite{CSVV}). 

We now consider a paracontact metric manifold  $(M,\varphi,\xi,\eta,g)$. By Theorem~\ref{main2} and equation $\varrho(\xi,\xi)=-2n+{\rm tr} h^2$, we get
\begin{align*}
	\bar\Delta \xi= Q\xi \quad  \Longleftrightarrow \quad Q\xi=-2n \xi \quad \Longleftrightarrow \quad  {\rm tr} h^2=0 \quad \textrm{and} \quad Q\xi \;\, \textrm{is collinear to} \,  \xi .
\end{align*}
%
Thus, we have the following result.
  
\begin{theorem}\label{6} 
 Let $(M,\eta ,\xi ,g,\varphi)$ be a paracontact metric manifold. Then, the following assertions are equivalent:
 
\medskip
 1) \, $Q\xi  = -2n \xi$; 

\smallskip 2) \,  $\xi$ is an infinitesimal harmonic transformation (equivalently, $1$-harmonic);

\smallskip 3) \, $M$ is $H$-paracontact and   {\rm tr}$h^2=0$.
\end{theorem}

\begin{remark}
{\em In general, a harmonic vector field needs not to be $1$-harmonic, nor conversely. This fact may be easily seen, for example, comparing the classifications of harmonic and $1$-harmonic left-invariant vector fields over three-dimen\-sional Lorentzian Lie algebras, given respectively in \cite{C3} and \cite{CSVV}. 

However, the above Theorem~\ref{6} yields that if the Reeb vector field of a paracontact metric manifold is $1$-harmonic, then it is harmonic, while the converse does not hold, because of the additional condition {\rm tr}$h^2=0$.}
\end{remark}

{

We already proved in Corollary~\ref{cor3.4} that any paracontact $(\kappa,\mu)$-space is $H$-paracontact. On the other hand, for a paracontact $(\kappa,\mu)$-space  one has $h^2=(k+1)\varphi^2$ (see for example \cite{CKM}), from which it easily follows that tr$h^2=0$ if and only if $k=-1$. Hence, by the above Theorem~\ref{6}, we have the following 

\begin{corollary}
The Reeb vector field of a paracontact $(\kappa,\mu)$-space is an infinitesimal harmonic transformation if and only if $\kappa=-1$. Whenever $\kappa \neq -1$, the Reeb vector field of a paracontact $(\kappa,\mu)$-space is harmonic but not $1$-harmonic.
\end{corollary}
}

Next, using the description of tensor $h$ given in equations \eqref{curv22}-\eqref{curv43}, we can easily deduce tr$h^2$ for all three-dimensional left-invariant paracontact metric structures classified in Theorem~\ref{mainTh}. Taking into account Theorems~\ref{6} and \ref{3Dharm}, we then get the following result.

\begin{corollary}\label{3Dinf}
The Reeb vector field of a three-dimensional homogeneous paracontact metric manifold is an infinitesimal harmonic transformation if and only if the manifold is isometric to one of the following cases, as classified in Theorem~{\em\ref{mainTh}}: 
\begin{itemize}
\item case {\it (2)} with $\beta=\gamma$; 
\vspace{-2mm}\item case {\it (3)};
\vspace{-2mm}\item case {\it (4)}; 
\vspace{-2mm}\item case {\it (5)}. 
\end{itemize}
\end{corollary} 

\noindent
{ The above Corollary~\ref{3Dinf} is compatible with the results about left-invariant Killing and $1$-harmonic vector fields on three-dimensional Lorentzian Lie groups obtained in \cite{CSVV} (see, in particular, Lemma~1, Theorem~6, Lemma~11 and Theorem~20 in \cite{CSVV}).}

\begin{remark}
{\em In the contact Riemannian case, the Reeb vector field is an infinitesimal harmonic transformation if and only if the contact Riemannian structure is $K$-contact \cite{Pe1}. 

Again by the description of tensor $h$ given in the previous Section, it is easily seen that $h=0$ (and so, the three-dimensional left-invariant paracontact metric structure is paraSasakian, see Theorem~2.2 in \cite{CIll}) if and only if we are either in case {\em (2)} with $\beta=\gamma$, in case {\em (3)}, or in case {\em (5)} with $\beta=0$. This corrects Theorem~4.3 in \cite{CIll}, as case {\em (a)} is not paraSasakian. \\
Comparing this classification with the above Corollary~\ref{3Dinf}, we see that in the following cases
\begin{itemize}
\vspace{-2mm}\item case {\em (3)};
\vspace{-2mm}\item case {\em (5)} with $\beta \neq 0$, 
\end{itemize}
$\xi$ is an infinitesimal harmonic deformation, although the paracontact metric structure is not $K$-paracontact. Thus, {\em  the class of paracontact metric structures, whose Reeb vector field is an infinitesimal harmonic transformation, is strictly larger than the one of $K$-paracontact structures}. 

{\smallskip
We can also exhibit a five-dimensional example of a paracontact, not $K$-paracontact metric manifold, whose characteristic vector field $\xi$ is an infinitesimal harmonic transformation. Consider the simply connected Lie group, whose Lie algebra $\mathfrak{g}={\rm Span}\{ \xi,X_1,X_2,Y_1,Y_2\}$ is described by
$$
\begin{array}{lll}
[X_1,X_2]=2X_2, & \quad [X_1,Y_1]=2\xi, & \quad [X_2,Y_1]=-2Y_2, \\[4 pt] 
[X_2,Y_2]=2(Y_1+\xi), & \quad [\xi, X_1]=-2 Y_1, & \quad [\xi,X_2]=-2Y_2 ,
\end{array}
$$
equipped with the left-invariant paracontact metric structure determined by the following conditions:
$$\varphi \xi=0, \quad \varphi X_i=X_i, \quad \varphi Y_i =-Y_i, \quad \eta(X_i)=\eta(Y_i)=0, \quad \eta(\xi)=1$$
and
$$g(X_i,X_j)=g(Y_i,Y_j)=0, \quad g(X_i,Y_j)=\delta_{ij},$$
for all $i,j=1,2$ (see \cite[Example 4.8]{CT}). As proved in \cite{CT}, this paracontact metric manifold is a paracontact $(\kappa,\mu)$-space, with $\kappa =-1$ and $\mu=2$. Hence, by Theorem~\ref{teo3.4}, it is $H$-paracontact. Moreover, $h^2=0$, although $hX_1 = -Y_1 \neq 0$. Therefore, this paracontact metric manifold is not $K$-paracontact, but by Theorem~\ref{6} its Reeb vector field is an infinitesimal harmonic transformation.

We also emphasize the fact that both in the above three-dimensional examples and in this five-dimensional example, with $\xi$ being an infinitesimal harmonic transformation, the tensor $h$ is two-step nilpotent.}
}\end{remark} 

The recent paper \cite{SS2} showed that the vector field $V$ determining a Riemannian Ricci soliton is necessarily an infinitesimal harmonic transformation. The same argument also applies to the semi-Riemannian case. A {\em Ricci soliton} is a semi-Riemannian manifold $(M,g)$, admitting a vector field $V$ and a real constant $\lambda$, such that
\begin{equation}\label{sol}
\varrho + \frac 12 L_V g = \lambda g.
\end{equation}
A Ricci soliton is said to be \emph{shrinking}, \emph{steady} or
\emph{expanding},  according to whether $\lambda>0$, $\lambda=0$ or $\lambda<0$,
respectively. An Einstein manifold, together with a Killing vector field, is a trivial solution of equation~\eqref{sol}. Ricci solitons have been intensively studied in recent years, particularly because of their relationship with the \emph{Ricci flow}. Examples and more details on Ricci solitons in semi-Riemannian settings may  be found in \cite{isr},\cite{CF} and references therein.

In analogy to the contact metric case, by a {\em paracontact (metric) Ricci soliton} we shall mean a paracontact metric manifold $(M,\varphi,\xi,\eta,g)$, such that equation~\eqref{sol} holds for $V=\xi$. In this case, we necessarily have
$$
Q\xi= \lambda \xi.
$$
In fact, if $V=\xi$, then equation \eqref{sol} yields
$$\begin{array}{rcl}
0 &=& \varrho (\xi,X) + \frac 12 \big( L_{\xi} g\big) (\xi,X) - \lambda g(\xi,X) \\[4 pt]
&=& g(Q\xi ,X)+\frac 12 g(\nabla_\xi\xi,X)+\frac 12 g(\nabla_X\xi , \xi)   -\lambda g(\xi,X) \\[4 pt]
&=& g(Q\xi ,X) -\lambda g(\xi,X),
\end{array}$$
for any vector field $X$, taking into account the fact that $\xi$ is unit and geodesic {(as it easily follows from \eqref{eq1})}.

On the other hand, if $(M,\varphi,\xi,\eta,g)$ is a paracontact Ricci soliton, then in particular $\xi$ is an infinitesimal harmonic transformation. Hence, Theorem~\ref{6} yields that $M$ is $H$-paracontact and $Q\xi  = -2n \xi$. Thus, $\lambda=  -2n$ and we have the following result.

\begin{theorem}\label{solit}
A paracontact Ricci soliton is $H$-paracontact, and is necessarily  shrinking.
\end{theorem}

In the contact Riemannian case, $\xi$ is an infinitesimal {harmonic} transformation only when it is Killing. As a consequence, a contact Riemannian Ricci soliton is necessarily trivial, that is, an Einstein $K$-contact metric manifold \cite{Pe1}. The above Theorem~\ref{solit} specifies that pseudo-Riemannian { paracontact} Ricci solitons must be found among $H$-{paracontact} manifolds. { On the one hand,  this does not exclude the existence of  nontrivial {paracontact} Ricci solitons, on the other hand, we could not find examples of nontrivial paracontact Ricci solitons.} { This leads to state the following

\smallskip\noindent
{\bf Open Question:} {\em There exist nontrivial paracontact Ricci solitons?}}

\begin{center}
\textsc{Dipartimento di Matematica e Fisica "E. De Giorgi", \\ Universit\`{a} del Salento, Lecce, Italy.} \\
\textit{E-mail address}: giovanni.calvaruso@unisalento.it, domenico.perrone@unisalento.it 
\end{center}

\end{document}